\documentclass{article}
\usepackage[utf8]{inputenc}
\usepackage{graphicx} 
\usepackage[page]{appendix}

\usepackage{amsmath}
\usepackage{amssymb}
\usepackage{amsfonts}
\usepackage{dsfont}

\usepackage[T1]{fontenc}   
\usepackage{lmodern}       
\usepackage{textcomp}      

\usepackage{amsthm}
\usepackage{mathtools}
\usepackage{enumerate}
\usepackage[hidelinks]{hyperref}
\usepackage{titlesec}
\usepackage{tikz}

\usepackage[shortlabels]{enumitem}
\usepackage{yfonts}
\usepackage{setspace}
\usepackage[utf8]{inputenc}
\usepackage[english]{babel}
\usepackage[margin = 1in]{geometry}
\usepackage{tikz}
\usepackage{graphicx}
\usepackage{authblk}
\usepackage{tikz-cd}
\usetikzlibrary{calc}
\usepackage{fancyhdr}
\usepackage{float}

\titleformat{\section}[block]{\centering\Large}{\thesection}{1em}{}

\theoremstyle{definition}

\theoremstyle{remark}
\newtheorem*{rem}{Remark}
\theoremstyle{plain}

\newtheorem{thm}{Theorem}[section]
\newtheorem{prop}[thm]{Proposition}
\newtheorem{lem}[thm]{Lemma}
\newtheorem{cor}[thm]{Corollary}

\newtheorem{defi}[thm]{Definition}

\title{The Four Periodic Inscription Problem}
\author{Ali Naseri Sadr}
\date{}

\begin{document}

\maketitle

\begin{abstract}
   We prove that four periodic curves in $\mathbb{C}$ inscribe a family of isosceles trapezoids for every angle $\theta\in(0,\pi)$, whose similarity types are determined by the signed areas between the quotients of the curves in the infinite cylinder. In the smooth case, the problem is identified with a Lagrangian intersection problem in $\mathbb{R}\times S^1\times\mathbb{C}$, while the continuous case follows by a convergence argument.  
\end{abstract}

\section{Introduction}
Consider four smooth embeddings of the real line $\gamma_1,\dots,\gamma_4\colon\mathbb{R}\to \mathbb{C}$ that satisfy the periodicity condition
\begin{equation*}
    \gamma_i(t+1) = \gamma_i(t)+\sqrt{-1}
\end{equation*}
for every $t$ and $1\leq i\leq 4$. Take the quotient of $\mathbb{C}$ under translation by $\sqrt{-1}$; the resulting space is the infinite cylinder $\mathbb{R}\times S^1$, and each $\gamma_i$ is mapped to a closed embedded curve $\tilde{\gamma_i}$ that generates $H_1(\mathbb{R}\times S^1)$. 
\begin{defi}
    We say the quadruple $(\gamma_1,\gamma_2,\gamma_3,\gamma_4)$ admits a balanced inscription of a quadrilateral $Q$ if there are four points $p_1,p_2,p_3,p_4$ in $\mathbb{C}$ with $p_i\in\gamma_i(\mathbb{R})$ for each $i$, and the quadrilateral formed by these points is similar to $Q$.
\end{defi}
Define
\begin{align*}
    &A(\gamma_1,\gamma_2) \coloneqq \int_{\tilde{\gamma}_1}xdy-\int_{\tilde{\gamma}_2}xdy,\\
    & A(\gamma_3,\gamma_4)\coloneqq \int_{\tilde{\gamma}_3}xdy-\int_{\tilde{\gamma}_4}xdy.
\end{align*}
These quantities are equal to the signed area between each pair of these curves in the infinite cylinder $\mathbb{R}\times S^1$. Tao conjectured in \cite[Conjecture 5.6]{1} that if 
\begin{equation}
\label{the one-half case of area inequality}
    A(\gamma_1,\gamma_2) + A(\gamma_3,\gamma_4) = 0,
\end{equation}
and
\begin{equation*}
    \gamma_1(\mathbb{R}) \cap \gamma_2(\mathbb{R})\cap \gamma_3(\mathbb{R})\cap \gamma_4(\mathbb{R})=\emptyset,
\end{equation*}
then the quadruple $(\gamma_1,\gamma_2,\gamma_3,\gamma_4)$ admits a balanced inscription of a square.
This is a variation of the Toeplitz square peg problem for a quadruple of periodic curves in $\mathbb{C}$. We prove this conjecture and generalize it to the case of isosceles trapezoids. We will show how the area relation in equation \eqref{the one-half case of area inequality} is connected to a well-known concept in symplectic geometry and Lagrangian Floer homology. 

An isosceles trapezoid is a trapezoid where the two non-parallel sides have equal length. The similarity type of an isosceles trapezoid is determined by a number $t\in (0,1)$ and an angle $\theta\in (0,\pi)$, where the ratio of the two parallel sides is $\frac{t}{1-t}$, and the angle between its two diagonals is $\theta$. We call such a quadrilateral a type $(t,\theta)$ trapezoid.
\begin{thm}
\label{Main Thm}
   Suppose $(\gamma_1,\gamma_2,\gamma_3,\gamma_4)$ are four smooth periodic embeddings of the real line in the complex plane, and the intersection of their images is empty. If these curves satisfy the generalized area relation
   \begin{equation}
   \label{t-area equality relation}
       t\cdot A(\gamma_1, \gamma_2) + (1-t)\cdot A(\gamma_3,\gamma_4) = 0
   \end{equation}
   for a number $t\in (0,1)$, then $(\gamma_1,\gamma_2,\gamma_3,\gamma_4)$ admits a balanced inscription of a type $(t,\theta)$ trapezoid for every $\theta\in (0,\pi)$.
\end{thm}
\begin{cor}[Conjecture 5.6 in \cite{1}]
\label{rectangle cor}
    If the area relation in \eqref{the one-half case of area inequality} holds, and $\gamma_1(\mathbb{R}) \cap \gamma_2(\mathbb{R})\cap \gamma_3(\mathbb{R})\cap \gamma_4(\mathbb{R})=\emptyset$, then $(\gamma_1,\gamma_2,\gamma_3,\gamma_4)$ inscribes a rectangle of every similarity type. In particular, they inscribe a square.
\end{cor}
\begin{rem}
    Theorem \ref{Main Thm} is the strongest result we can get for a quadruple of periodic curves. For instance, if four vertical lines in the plane do not satisfy the area relation in \eqref{the one-half case of area inequality}, then there are no rectangles inscribed in these four curves. 
\end{rem}
Theorem \ref{Main Thm} can be generalized to continuous curves if we can define the quantities $A(\gamma_1,\gamma_2)$ and $A(\gamma_3,\gamma_4)$ such that they behave well under approximation by smooth curves. For this, we need to assume each curve has zero Lebesgue measure. Consider a continuous periodic embedding $\gamma\colon\mathbb{R}\to \mathbb{C}$; its image $\tilde{\gamma}$ divides $\mathbb{R}\times S^1$ into two regions by Jordan curve theorem. Let $U_{\tilde{\gamma}}^{-}$ be the one that contains $-\infty$, and $U_{\tilde{\gamma}}^{+}$ be the one that contains $+\infty$. For a continuous periodic pair $(\gamma_1,\gamma_2)$ in $\mathbb{C}$, we define
\begin{equation}
\label{signed area for continuou curves}
   A(\gamma_1,\gamma_2)\coloneqq \int_{\mathbb{R}\times S^1}(\mathds{1}_{U_{\tilde{\gamma}_1}^{-}}-\mathds{1}_{U_{\tilde{\gamma}_2}^{-}})d\mu
\end{equation}
where $\mu$ denotes the Lebesgue measure on $\mathbb{R}\times S^1$. If each $\gamma_i$ has zero Lebesgue measure, and we have two sequences of smooth periodic curves $\gamma_1^n\to\gamma_1$ and $\gamma_2^n\to\gamma_2$, then $A(\gamma_1^n, \gamma_2^n)$ converges to $A(\gamma_1,\gamma_2)$.  
\begin{thm}
\label{main thm for continuous curves}
    Suppose $(\gamma_1,\gamma_2,\gamma_3,\gamma_4)$ are four continuous periodic embeddings of the real line in the complex plane with zero Lebesgue measure, and the intersection of their images is empty. Moreover, assume that they satisfy equation \eqref{t-area equality relation} for a number $t\in (0,1)$, where we use the definition in \eqref{signed area for continuou curves} for $A(\gamma_1,\gamma_2)$ and $A(\gamma_3,\gamma_4)$. Then $(\gamma_1,\gamma_2,\gamma_3,\gamma_4)$ admits a balanced inscription of a type $(t,\theta)$ trapezoid for every $\theta\in (0,\pi)$.
\end{thm}
We prove Theorem \ref{main thm for continuous curves} by a standard convergence argument, and we prove Theorem \ref{Main Thm} using symplectic geometry. We will show how balanced inscriptions of type $(t,\theta)$ trapezoids in a smooth quadruple $(\gamma_1, \gamma_2,\gamma_3,\gamma_4)$ can be identified with a Lagrangian intersection problem for two non-compact Lagrangian infinite cylinders in a symplectic manifold that depends on $t$. 
We then use Lagrangian Floer homology to show that this intersection is non-empty if the area relation in \eqref{t-area equality relation} holds. Since these Lagrangians are non-compact, we need to control their behavior at infinity in order to compactify the moduli spaces of holomorphic strips with boundary conditions on these Lagrangians; this requires a notion of \textit{suitability} that is based on the ideas used by Oh in \cite{5}, Seidel and Abouzaid in \cite{MR2602848}[Section $7$], and by Ritter in \cite{RittPaper}. See \cite{3,8, RectanglePaper} for earlier applications of symplectic geometry to inscription problems for Jordan curves in the plane and \cite{NASERISADR_2026} for the case of Jordan curves in surfaces of constant curvature.

The paper is organized as follows. We define our symplectic manifold and the Lagrangians associated to a quadruple of smooth periodic curves in Section \ref{sec: symp setting}. Then we briefly review Lagrangian Floer homology and set up the machinery needed to compute the Floer homology in Section \ref{sec: Lag Floer}. In Sections \ref{sec: vert lines and clean} and \ref{sec: periodic curves and suitable lags}, we compute the Floer homology, and show it is non-trivial. We prove our main theorems, and finish with a couple of inscription results for the case of three periodic curves in Section \ref{sec: main section}.

\section*{Acknowledgments}
The author is grateful to his advisors, John Baldwin and Josh Greene, for their invaluable guidance, support, and insightful conversations about this work.

\section{Symplectic Setting}
\label{sec: symp setting}
Fix a positive number $t$ less than $1$. We equip $\mathbb{C}^2$ with the symplectic form
\begin{equation}
    \omega = t\cdot dx_1\wedge dy_1+(1-t)\cdot dx_2\wedge dy_2.
\end{equation}
Consider a Hamiltonian $F$ on $\mathbb{C}^2$ defined by
\begin{equation}
\label{Defintion of Our Major Hamiltonian}
    F(z,w) = \frac{t(1-t)\Vert z-w\Vert^2}{2}.
\end{equation}
The Hamiltonian vector field of $F$ with respect to $\omega$ is determined by
\begin{equation*}
    \iota_{X_F}\omega = -dF.
\end{equation*}
We have
\begin{equation*}
    dF= t(1-t)\Big((x_1-x_2)dx_1+(x_2-x_1)dx_2+(y_1-y_2)dy_1+(y_2-y_1)dy_2\Big).
\end{equation*}
Hence, we get
\begin{equation*}
    X_F = (1-t)\Big((y_2-y_1)\partial x_1+(x_1-x_2)\partial y_1\Big) + t\Big((y_1-y_2)\partial x_2+(x_2-x_1)\partial y_2\Big).
\end{equation*}
The Hamiltonian flow of $F$ is given by
\begin{equation}
\label{the flow equation}
    \varphi_\theta(z,w) = (tz+(1-t)w+e^{i\theta}(1-t)(z-w), tz+(1-t)w + e^{i\theta}t(w-z))
\end{equation}
for $\theta\in [0,2\pi]$.
Notice that if we have a pair of distinct points $z$ and $w$ in $\mathbb{C}$, then the components of $\varphi_{\theta}(z,w)$ are the two opposite vertices of an isosceles trapezoid with the line segment going from $z$ to $w$ being one of its diagonals. The similarity type of this trapezoid is determined by the pair $(t,\theta)$. We call this a trapezoid of type $(t,\theta)$; see the following figure.

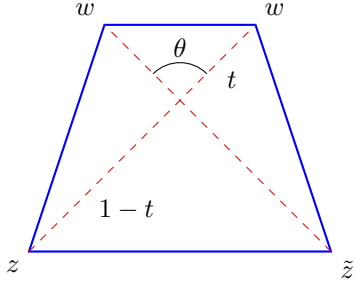
\begin{figure}[h]
    \begin{center}
\begin{tikzpicture}[scale=1]
  \coordinate (A) at (0,0);    
  \coordinate (B) at (4,0);    
  \coordinate (C) at (3,3);    
  \coordinate (D) at (1,3);    
  
  \node[above right] at (C) {$w$};
  \node[above left]  at (D) {$\Tilde{w}$};
  \node[below left]  at (A) {$z$};
  \node[below right] at (B) {$\Tilde{z}$};
  
  \coordinate (O) at (2,2);
  
  \draw[thick, blue] (A) -- (B) -- (C) -- (D) -- cycle;
  
  \draw[dashed, red] (A) -- (C);
  \draw[dashed, red] (B) -- (D);
  
  \node[below right] at ($ (A)!0.4!(O) $) {$1-t$};
  \node[below right] at ($ (O)!0.5!(C) $) {$t$};
  
  \draw (O) ++(45:0.5) arc (45:135:0.5);
  
  \node at ($(O)+(90:0.7)$) {$\theta$};
  
\end{tikzpicture}
    \end{center}
    \caption{$\varphi_\theta(z,w)=(\Tilde{z}, \Tilde{w})$}
    \label{fig1}
\end{figure}
Fix an angle $\theta\in(0,\pi)$ in the following. We will need the following for later compactness arguments.

\begin{lem}
\label{bounded band similarity type}
Consider a positive number $N$ and $(z,w)\in \mathbb{C}^2$. Let $(\Tilde{z}, \Tilde{w})=\varphi_\theta(z,w)$, and suppose 
\begin{align*}
    & \vert \operatorname{Re}(z)\vert\leq N, \vert \operatorname{Re}(\Tilde{z})\vert\leq N,\\
    & \vert \operatorname{Re}(w)\vert\leq N, \vert \operatorname{Re}(\Tilde{w})\vert\leq N.
\end{align*}
Then there exists a constant $b(N,t,\theta)$ depending only on $N,t$, and $\theta$ so that 
\begin{equation*}
    \Vert z-w\Vert\leq b(N,t,\theta).
\end{equation*}
\end{lem}

\begin{proof}
    Let $l$ in $[-N,N]\times\mathbb{R}$ be the line segment going through $tz+(1-t)w$ and parallel to the real line in $\mathbb{C}$. This point divides this segment into two parts; suppose $d_1$ is length of the section to the right of $tz+(1-t)w$ and $d_2$ is length of the left one. It is clear that both $d_1$ and $d_2$ are less than $2N$. If one of the points $w$ or $\tilde{w}$ lies on $l$, then we get
    \begin{equation*}
        t\Vert z-w\Vert<\max(d_1, d_2)\leq 2N.
    \end{equation*}
    Therefore, we can assume none of them is on $l$. Now suppose these points are on different sides of $l$, and let $\varphi_1$ denote the angle between $l$ and the line going through $tz+(1-t)w$ and $w$. Define $\varphi_2$ for $\tilde{w}$ in a similar way. We have $\varphi_1+\varphi_2=\theta$ and assume both $\varphi_1$ and $\varphi_2$ are not equal to $\frac{\pi}{2}$. We get
    \begin{align}
    \label{line eq1}
        &t\Vert z-w\Vert<\frac{\max(d_1, d_2)}{\cos(\varphi_1)}\leq\frac{2N}{\cos(\varphi_1)},\\
        \label{line eq2}
        & t\Vert z-w\Vert<\frac{\max(d_1, d_2)}{\cos(\varphi_2)}\leq\frac{2N}{\cos(\varphi_2)}
    \end{align}
    because we assumed both $w$ and $\tilde{w}$ are inside the strip. We conclude that
    \begin{equation*}
        \Vert z-w\Vert<\frac{2N}{t\cos(\frac{\theta}{2})}
    \end{equation*}
    since $\varphi_1+\varphi_2=\theta$. If one of the $\varphi_i's$ is equal to $\frac{\pi}{2}$, then the other one is $\theta-\frac{\pi}{2}$ and this proves
    \begin{equation*}
        \Vert z-w\Vert<\frac{2N}{t\sin(\theta)}.
    \end{equation*}
    Finally, suppose both of the points are on one side of $l$ and define the angles $\varphi_i$'s as in the previous case. In this case, we have $\varphi_1+\varphi_2=\pi-\theta$. We get the same inequalities as in \eqref{line eq1} and \eqref{line eq2}; hence,
    \begin{equation*}
        \Vert z-w\Vert<\frac{2N}{t\sin(\frac{\theta}{2})}.
    \end{equation*}
    This concludes the proof with 
    \begin{equation*}
        b(N, t, \theta)=\frac{2N}{t\cdot \min\big(\cos(\frac{\theta}{2}), \sin(\theta), \sin(\frac{\theta}{2})\big)}.
    \end{equation*}
\end{proof}

Let $\gamma_1,\gamma_2,\gamma_3,\gamma_4$ be four smooth periodic curves in $\mathbb{C}$. Then the symplectic form $dx\wedge dy$ vanishes on each of these curves, so we get two Lagrangians $\gamma_1\times\gamma_3$ and $\gamma_2\times\gamma_4$ in $(\mathbb{C}^2, \omega)$. We denote the balanced inscriptions of a type $(t,\theta)$ trapezoid in the quadruple $\gamma_1,\gamma_2,\gamma_3,\gamma_4$ by $\mathcal{I}(\gamma_1,\gamma_2,\gamma_3,\gamma_4)$. 
Observe that the only fixed points of $\varphi_\theta$ are the diagonal points $(z,z)$ in $\mathbb{C}^2$.
Hence, there is a one-to-one correspondence between the union of the balanced inscriptions with $\gamma_1(\mathbb{R})\cap\gamma_2(\mathbb{R})\cap \gamma_3(\mathbb{R})\cap \gamma_4(\mathbb{R})$ and the Lagrangian intersection
\begin{equation*}
    \gamma_2\times\gamma_4\cap \varphi_\theta(\gamma_1\times\gamma_3)
\end{equation*}
by equation \eqref{the flow equation}. We have a translation map $T$ on $\mathbb{C}^2$ given by
\begin{equation*}
    T(z,w) = (z+\sqrt{-1}, w+\sqrt{-1}).
\end{equation*}
This is a symplectomorphism with respect to $\omega$, and we have
\begin{equation*}
    F\circ T= F.
\end{equation*}
Let $X$ denote the quotient of $\mathbb{C}^2$ by $T$; the symplectic form $\omega$ reduces to a symplectic form on $X$ which we also denote by $\omega$ without loss of generality. The Hamiltonian $F$ is invariant under $T$, so it defines a Hamiltonian on $(X,\omega)$ that we denote by $F$ too. Furthermore, let $\varphi_\theta$ be the time $\theta$ flow of $F$ in $(X,\omega)$ as well. We have the quotient map $\pi\colon\mathbb{C}^2\to X$. Let $j$ be the standard complex structure on $\mathbb{C}^2$; since $T$ is holomorphic, $j$ descends to a complex structure that we also denote by $j$.

We assumed each $\gamma_i$ is a periodic curve, so we get two embedded Lagrangians in $(X,\omega)$ for each of $\gamma_1\times\gamma_3$ and $\gamma_2\times\gamma_4$. We denote these Lagrangians by $L_{13}$ and $L_{24}$ respectively, and define $L_{13}^{\theta}\coloneqq \varphi_\theta(L_{13})$. Hence, we get
\begin{equation}
\label{equality of intersections}
    \frac{\mathcal{I}(\gamma_1,\gamma_2,\gamma_3,\gamma_4)\bigcup (\gamma_1(\mathbb{R})\cap \gamma_2(\mathbb{R})\cap \gamma_3(\mathbb{R})\cap \gamma_4(\mathbb{R}))}{\text{Translation by $\sqrt{-1}$}} = L_{24}\cap L_{13}^\theta.
\end{equation}
\begin{rem}
    Note that $(X,\omega)$ is symplectomorphic to the cotangent bundle of an infinite cylinder, and each of the $L_{ij}$'s is a Lagrangian $S^1\times\mathbb{R}$ in $X$.  
\end{rem}
Our goal is to show the set $\mathcal{I}(\gamma_1,\gamma_2,\gamma_3,\gamma_4)$ is non-empty if the area formula holds, and the intersection of our curves is empty. Hence, it suffices to show the Lagrangian intersection $L_{24}\cap L_{13}^\theta$ is non-empty under the area criterion. We will prove this using Lagrangian Floer homology for a modified version of the pair $(L_{24}, L^\theta_{13})$. In order to define the Floer homology, we need the intersection in equation \eqref{equality of intersections} to be compact.
\begin{lem}
\label{compactness of intersections}
    Suppose $\gamma_1,\gamma_2,\gamma_3,\gamma_4$ are four periodic curves in $\mathbb{C}$. Then the intersection $L_{24}\cap L_{13}^{\theta}$ is compact.
\end{lem}
\begin{proof}
    Since the curves are periodic, we can find a sufficiently large $N$ such that each $\gamma_i$ is inside the strip $[-N,N]\times \mathbb{R}$. Now suppose $(z,w)$ are two opposite vertices of a type $(t,\theta)$ trapezoid inscribed in $\gamma_i$'s. By Lemma \ref{bounded band similarity type}, $\Vert z-w\Vert$ is bounded by a constant $b(N,t,\theta)$. Consider a point $p\in L_{24}\cap L_{13}^{\theta}$; we can assume $p=\pi(z,w)$ with $z\in [-N,N]\times [0,1]$ because $\gamma_2$ is a periodic curve. Thus we get
    \begin{equation*}
        \Vert z\Vert\leq \sqrt{N^2+1},\hspace{1mm} \Vert w\Vert\leq \sqrt{N^2+1}+b(N,t,\theta).
    \end{equation*}
    We conclude that
    \begin{equation*}
        L_{24}\cap L_{13}^{\theta}\subset \pi(B_{\sqrt{N^2+1}}\times B_{\sqrt{N^2+1}+b(N,t,\theta)}).
    \end{equation*}
\end{proof}
Suppose $L$ is one of the Lagrangians $L_{24}$ or $L_{13}^\theta$. It is crucial for our construction to understand the map induced by inclusion on the fundamental groups $\pi_1(L)\to \pi_1(X)$. In particular, we have the following.
\begin{lem}
\label{Inclusion map induces iso on fundamental group}
    Let $\iota\colon L\to X$ denote the inclusion map. Then the map $\iota_{\#}\colon\pi_1(L)\to \pi_1(X)$ is an isomorphism, and $\pi_2(X,L)$ is trivial.
\end{lem}
\begin{proof}
    We first prove the lemma for $L_{24}$. Note that $\mathbb{C}^2$ is the universal cover of $X$ and the covering map is $\pi$. We know $\pi^{-1}(L_{24})=\gamma_2\times\gamma_4$ because the curves are periodic. It follows that $\gamma_2\times\gamma_4$ is the universal cover of $L_{24}$, and the covering map is the restriction of $\pi$ to $\gamma_2\times\gamma_4$. This proves the first claim for $L_{24}$. Similarly, one can show the induced map on $\pi_1(L_{13})\to\pi_1(X)$ is an isomorphism, and this proves the claim for $L_{13}^\theta$ as well since it is a Hamiltonian deformation of $L_{13}$. For the second point, consider the following exact sequence
    \begin{equation*}
        \pi_2(X)\to \pi_2(X,L)\to \pi_1(L)\to \pi_1(X).
    \end{equation*}
    The last map is an isomorphism, and $\pi_2(X)$ is trivial, so $\pi_2(X,L)$ is trivial.
\end{proof}

We have the complex structure $j$ on $X$. Define a function $g$ from $\mathbb{C}^2\to \mathbb{R}$ by 
\begin{equation}
\label{g equation}
g(z,w) = t\operatorname{Re}(z) + (1-t)\operatorname{Re}(w)
\end{equation}
This function is invariant under $T$, so we get a function on $X$ which we denote by $g$ as well. Notice that the function 
\begin{equation}
\label{The proper exhaustion}
    r(x) = F(x) + g^2(x)
\end{equation}
is a proper function on $X$. Moreover, we have
\begin{equation}
\label{g does not change under phi}
    g\circ\varphi_{\theta}(z,w) = g(z,w).
\end{equation}
We conclude this section with the following lemma.
\begin{lem}
\label{The functions are PSH}
    The functions $F$ and $g^2$ are plurisubharmonic on $(X,j)$.
\end{lem}
\begin{proof}
    Since $\pi\colon\mathbb{C}^2\to X$ is a local diffeomorphism and holomorphic, it suffices to prove the corresponding lifts of these functions on $\mathbb{C}^2$ are plurisubharmonic. Consider a holomorphic map $f\colon\Omega\to\mathbb{C}^2$ defined on an open domain $\Omega\subset\mathbb{C}$; we can write $f=(f_1,f_2)$ for two analytic functions $f_1$ and $f_2$. 
    Define $g_1, g_2\colon\mathbb{C}^2\to \mathbb{R}$ by
    \begin{equation}
        g_1(z,w) = \operatorname{Re}(z),\hspace{2mm} g_2(z,w) = \operatorname{Re}(w).
    \end{equation}
    It suffices to prove each of the functions $F\circ f, g_1^2\circ f$, and $g_2^2\circ f$ are subharmonic. We have
    \begin{equation*}
        F\circ f = \frac{t(1-t)\Vert f_1-f_2\Vert^2}{2}.
    \end{equation*}
    We can write $f_1-f_2 = (u,v)$ where $u$ and $v$ are real harmonic functions on $\Omega$ because $f_i$'s are analytic. Hence, we get
    \begin{align*}
        &\Delta (\Vert f_1-f_2\Vert^2) = \Delta u^2+\Delta v^2 = \\2(u_x^2+uu_{xx})+2(u_y^2+uu_{yy})+
        & 2(v_x^2+vv_{xx})+2(v_y^2+vv_{yy}) = 2(u_x^2+u_y^2+v_x^2+v_y^2)\geq 0 .
    \end{align*}
    Similarly, if we write $f_1 = (u, v)$, we have
    \begin{equation*}
        g_1^2\circ f_1 = u^2,
    \end{equation*}
    and we get
    \begin{equation*}
        \Delta(u^2) = (u_x^2+uu_{xx})+2(u_y^2+uu_{yy})=2(u_x^2+u_y^2)\geq 0.
    \end{equation*}
    The case for $g_2^2\circ f$ can be proved in the same manner.
\end{proof}

\section{Lagrangian Floer Homology}
\label{sec: Lag Floer}

For a brief discussion of Lagrangian Floer homology, see \cite{11}.
We will work with mod $2$ coefficients for Floer homology in the following. We will construct Lagrangian Floer homology for a pair of non-compact Lagrangians $(L_1,L_2)$ in $X$.
In order to define Lagrangian Floer homology for a pair $(L_1, L_2)$, one needs to show holomorphic strips with boundary on these two Lagrangian have bounded image. Since the two Lagrangians are non-compact, this is only possible if we can control the behavior of $L_1$ and $L_2$ at infinity; see \cite{RittPaper,5} for more details.
The one form 
\begin{equation}
    \lambda = t\cdot x_1dy_1+(1-t)\cdot x_2dy_2
\end{equation}
on $\mathbb{C}^2$ is invariant under the action by $T$. Hence, this descends to a one form on $X$, and we have $d\lambda = \omega$.
For a Lagrangian $L$ in $X$, $\lambda|_{L}$ is a closed one form, and it gives us an element in $H^1_{dR}(L)$; we call this the area class of $L$.
Let $(L_1, L_2)$ be a pair of Lagrangians in $(X,\omega)$, and define the path space $\mathcal{P}(L_1,L_2)$ by
\begin{equation*}
    \{x\in W^{1,2}([0,1], X)\hspace{1mm}|\hspace{1mm} x(0)\in L_1,\hspace{1mm} x(1)\in L_2\}.
\end{equation*}
Fix a Hamiltonian $H\colon X\times[0,1]\to\mathbb{R}$. The symplectic action one form on $\mathcal{P}(L_1, L_2)$ is defined by
\begin{equation*}
    \omega_{H}(\xi) = \int_0^1 \omega(\xi, \Dot{x}(t)-X_{H}(x(t),t))dt
\end{equation*}
for every $x\in \mathcal{P}(L_1, L_2)$ and $\xi\in T_x\mathcal{P}(L_1, L_2)$. This is a closed one form on the path space between two Lagrangians, but it is not exact in general. 
\begin{lem}
\label{Condition of Action being exact}
    Suppose $(L_1, L_2)$ is a pair of Lagrangians in $X$, and $H$ is a time-dependent Hamiltonian on $X$. Furthermore, assume the inclusion map $\iota\colon L_i\to X$ induces an isomorphism on the fundamental group for each $i$. 
    Then $\omega_H$ is exact iff the following diagram commutes.
    \[
\begin{tikzcd}
H_1(L_1) \arrow[rr, "h"] \arrow[dr, "\lambda"'] & & H_1(L_2) \arrow[dl, "\lambda"] \\
& \mathbb{R} &
\end{tikzcd}
\]
where the map $h$ is defined via the inclusion isomorphisms $\iota\colon H_1(L_i)\to H_1(X)$.
\end{lem}
\begin{proof}
    Take an arbitrary closed path $u\colon S^1\to \mathcal{P}(L_1, L_2)$; it suffices to show $\int_u\omega_H=0$. This is a cylinder $u\colon S^1\times[0,1]\to X$ with one end on $L_1$ and the other end on $L_2$. We denote each of these ends by $\beta_i$ respectively for $i=1,2$. Observe that $u$ is a homotopy between $\beta_1$ and $\beta_2$ in $X$, so we get $\beta_2=h(\beta_1)$
    by definition. We have
    \begin{equation*}
        0=\omega_{H}[u] =\int_{S^1}\int_0^1\omega(u_s, u_t-X_{H}(u,t))dtds =
   \int_u\omega-\int_0^1\int_{S^1}\frac{\partial H(u,t)}{\partial s}dsdt = \int_u\omega = \lambda(\beta_2)-\lambda(\beta_1),
    \end{equation*}
    where in the last step, we applied Stokes' theorem.
\end{proof}
\begin{prop}
\label{The area Lagrangian proposition}
    Let $(\gamma_1,\gamma_2,\gamma_3,\gamma_4)$ be a quadruple of smooth periodic curves in $\mathbb{C}$, and consider the corresponding Lagrangians $L_{24}$ and $L_{13}^\theta$ in $X$. Then the symplectic action one form on $\mathcal{P}(L_{24}, L_{13}^\theta)$ is exact iff 
    \begin{equation}
    \label{The area Lagrangian equation}
        t(\int_{\Tilde{\gamma}_1}xdy-\int_{\Tilde{\gamma}_2}xdy)+(1-t)(\int_{\Tilde{\gamma}_3}xdy-\int_{\Tilde{\gamma}_4}xdy) = 0.
    \end{equation} 
\end{prop}
\begin{proof}
    We consider a curve $\beta_1$ that generates $H_1(L_{24})$, and $\beta_2=h(\beta_1)$ a generator of $H_1(L_{13})$ where $h$ is the map induced by the inclusion maps of $L_{24}$ and $L_{13}$ in $X$. Each $\gamma_i$ is an embedding of $\mathbb{R}$ into $\mathbb{C}$, and we define $\beta_1\colon [0,1]\to L_{24}$ by
    \begin{equation*}
        \beta_1(s) = \pi(\gamma_2(s), \gamma_4(s))
    \end{equation*}
    for $0\leq s\leq 1$. This is a closed path because $\gamma_i(s+1) = \gamma_i(s)+\sqrt{-1}$. The curve $\beta_2$ in $L_{13}$ is defined similarly, and we have $h(\beta_1) = \beta_2$. We get
    \begin{equation*}
        \lambda(\beta_1) = t\int_0^1x(\gamma_2(s))\dot{y}(\gamma_2(s))ds + (1-t)\int_0^1x(\gamma_4(s))\dot{y}(\gamma_4(s))ds=t\int_{\Tilde{\gamma}_2}xdy+(1-t)\int_{\Tilde{\gamma}_4}xdy.
    \end{equation*}
    Analogously, we compute
    \begin{equation*}
        \lambda(\beta_2) = t\int_0^1x(\gamma_1(s))\dot{y}(\gamma_1(s))ds + (1-t)\int_0^1x(\gamma_3(s))\dot{y}(\gamma_3(s))ds=t\int_{\Tilde{\gamma}_1}xdy+(1-t)\int_{\Tilde{\gamma}_3}xdy.
    \end{equation*}
    Thus the claim follows for $\mathcal{P}(L_{24}, L_{13})$ by Lemma \ref{Condition of Action being exact}. The same holds for $\mathcal{P}(L_{24}, L_{13}^\theta)$ since $L_{13}^\theta$ is a Hamiltonian perturbation of $L_{13}$. 
\end{proof}
\begin{defi}
\label{Suitable pair Defintion}
    Consider a pair of non-compact Lagrangians $(L_1, L_2)$ in $X$. We say that these Lagrangians are suitable if they satisfy the following assumptions
    \begin{enumerate}
        \item The inclusion map $\iota\colon L_i\to X$ induces an isomorphism on $\pi_1(L_i)$ for each $i$.
        \item The intersection $L_1\cap L_2$ is compact. 
        \item The symplectic action one form on $\mathcal{P}(L_1, L_2)$ is exact.
        \item The function $g$ restricted to $L_i$ is bounded for each $i$ where $g$ is defined in \eqref{g equation}.
        \item There exists a positive constant $A$ such that if 
\begin{equation}
\label{the Hopf Lemma Ess Assumption}
    (x,v)\in (TL_1\cup TL_2)\cap \ker(dF)
\end{equation}
with $F(x)> A$, then $jv\in \ker(dF)$ where $F$ is the plurisubharmonic function defined in \eqref{Defintion of Our Major Hamiltonian}.
\end{enumerate}
\end{defi}
In the following, we assume $(L_1, L_2)$ is a pair of suitable Lagrangians in $X$.
Consider a primitive for $\omega_{H}$ on $\mathcal{P}(L_1, L_2)$, and denote it by $\mathcal{A}_{H}$; we call this the symplectic action functional.  
Critical points of $\mathcal{A}_{H}$ are in one-to-one correspondence with Hamiltonian chords of $H$ that start from $L_1$ and ends on $L_2$. We can identify these chords with
the set $L_1\cap (\phi_{H}^{1})^{-1}(L_2)$ where $\phi_{H}^{1}$ denotes the Hamiltonian flow of $H$ at $t=1$. Now consider a family of $\omega$-compatible time dependent almost complex structures $\{J_t\}_{0\leq t\leq 1}$ on $X$. This family induces a metric on $\mathcal{P}(L_1, L_2)$, and the gradient flow lines of $\mathcal{A}_{H}$ are in one-to-one correspondence with the solutions of

\begin{equation}
\label{CRF}
    u_s + J_t(u)(u_t-X_{H}(u))=0,
\end{equation}
$u\colon\mathbb{R}\times[0,1]\to X$,
with boundary conditions 
\begin{equation}
    u(s,0) \in L_1 \hspace{1mm} u(s,1)\in L_2.
\end{equation}
This is called the perturbed Cauchy-Riemann equation (with respect to $\{J_t\}_{0\leq t\leq 1}$ and $H$) with boundary conditions on the pair $(L_1, L_2)$; we refer the unfamiliar reader to \cite[Section 1.3]{12}.
If $u(s,t)$ is a solution to \eqref{CRF}, then $u(s+s_0,t)$ is a solution for every $s_0$ as well.
For a solution $u$ to \eqref{CRF}, we define its energy by
\begin{equation*}
    E(u) = \int_{-\infty}^{\infty}\int_0^1 \Vert u_s\Vert^2 dtds
\end{equation*}
where $\Vert u_s\Vert^2$ is measured with respect to the Riemannian metric induced by $J_t$. If we have a solution with $E(u)<\infty$, then one can show there are $x,y\in L_1$ with
\begin{equation*}
    \lim_{s\to -\infty} u(s,t) = \phi_H^t(x) \hspace{1mm} \lim_{s\to +\infty} u(s,t) = \phi_H^t(y),
\end{equation*}
$\phi_H^1(x)\in L_2$, $\phi_H^1(y)\in L_2$. If we write $x(t) = \phi_H^t(x)$ and $y(t)=\phi_H^t(y)$ for $t\in [0,1]$, then  
\begin{equation}
\label{Action-Energy}
    E(u) = \mathcal{A}_{H}(y(t)) - \mathcal{A}_{H}(x(t)). 
\end{equation}

\begin{defi}
    We say a time-dependent Hamiltonian $H$ and a family of almost complex structures $\{J_t\}_{0\leq t\leq 1}$ on $X$ are admissible for a suitable pair $(L_1, L_2)$ if the following conditions hold.
    \begin{enumerate}
        \item $\phi_{H}^1(L_1)$ meets $L_2$ transversely, and $H$ is compactly supported.
        \item There exists a compact set $K\subset X$ such that $J_t$ agrees with $j$ on $X\setminus K$ for every $t$.
        \item The moduli spaces of finite energy solutions to \eqref{CRF} with boundary conditions on $(L_1, L_2)$ are smooth manifolds. 
    \end{enumerate}
\end{defi}

It is a standard fact in Floer theory that a generic choice of $(H, J_t)$ is admissible; we refer the unfamiliar reader to \cite{TrinFloerTheory} for more details.
Now consider a suitable pair of Lagrangians $(L_1,L_2)$ in $X$, and an admissible $(H, J_t)$ for this pair. Let $CF(L_1, L_2, H, J_t)$ be the vector space generated by the intersection points in\break 
$L_1\cap (\phi_{H}^{1})^{-1}(L_2)$ over $\mathbb{F}_2$. 
For a pair of points $(x,y)$ in this intersection, denote the moduli space of finite energy solutions to \eqref{CRF} that go from $x(t)$ to $y(t)$
by $\mathcal{M}(x,y)$, and define the unparametrized moduli space by
\begin{equation*}
    \widehat{\mathcal{M}}(x,y)= \mathcal{M}(x,y)/\mathbb{R} 
\end{equation*}
where we quotient by translation in the $s$ direction.
Define a morphism $\partial\colon CF(L_1, L_2, H, J_t)\to CF(L_1, L_2, H, J_t)$ by
\begin{equation*}
    \partial x = \sum_{y}\#\widehat{\mathcal{M}}_0(x,y)\cdot y
\end{equation*}
where $\widehat{\mathcal{M}}_0(x,y)$ is the zero dimensional part of the unparametrized moduli spaces  that go from $x$ to $y$.
\begin{prop}
\label{boundary squared is zero}
    If $(L_1, L_2)$ is a suitable pair in $X$, then $\partial^2 = 0$ for an admissible $(H, J_t)$.
\end{prop}
\begin{proof}
    First, we show every holomorphic strip $u$ with boundary on $(L_1, L_2)$ and finite energy has 
    \begin{equation*}
        E(u)\leq C(L_1, L_2, H)
    \end{equation*}
    where $C$ depends only on $(L_1, L_2)$ and $H$. Let $(x(t),y(t))$ be the two ends of $u$; then we have
    \begin{equation*}
        E(u) = \mathcal{A}_H(y(t)) - \mathcal{A}_H(x(t)).
    \end{equation*}
    Note that $\mathcal{I}\coloneqq L_1\cap(\phi_H^1)^{-1}(L_2)$ is compact because $L_1\cap L_2$ is compact, and $H$ is compactly supported. Hence,
    \begin{equation*}
        E(u)\leq 2\cdot\sup_{\mathcal{I}}\vert\mathcal{A}_H\vert=C(L_1,L_2,H).
    \end{equation*}
    The next step is to show that a holomorphic strip $u$ stays inside a compact set that only depends on $(L_1, L_2)$ and $(H, J_t)$. Let $K$ be the closure of
    \begin{equation*}
        \{x\in X\colon \exists t \hspace{1mm} \text{with} \hspace{1mm} J_t(x)\neq j\}\bigcup \text{supp}(H)\bigcup \mathcal{I}.
    \end{equation*}
    This is a compact set because of our choices. Note that if $\Omega\coloneqq u^{-1}(X\setminus K)$ is not empty, then $u$ restricted to $\Omega$ satisfies
    \begin{equation*}
        u_s + ju_t = 0
    \end{equation*}
    which is the ordinary Cauchy-Riemann equation.
    Suppose $u$ is a holomorphic strip, and $\Omega\neq \emptyset$. Then $g^2\circ u$ is a subharmonic function on $\Omega$ by Lemma \ref{The functions are PSH}. Hence,
    \begin{equation*}
        \sup_{\Omega} g^2\circ u\leq\sup_{\partial\Omega}g^2\circ u\leq \max(\sup_{L_1}g^2, \sup_{L_2} g^2, \sup_{K}g^2)\coloneqq M(L_1, L_2, K).
    \end{equation*}
    In the last step, we used the fact that $g$ is bounded on the pair $(L_1, L_2)$. Similarly, $F\circ u$ is subharmonic by Lemma \ref{The functions are PSH}. Let $A$ be the constant determined by the last assumption of Definition \ref{Suitable pair Defintion}. We know $F\circ u$ does not have an interior maximum point, and assume $p\in\partial\Omega$ is a boundary maximal point. We claim 
    \begin{equation*}
        u(p)\in K\bigcup F^{-1}[0,A].
    \end{equation*}
    By contradiction, assume this is not the case; then we must have $u(p)\in L_1$ or $u(p)\in L_2$. Without loss of generality, suppose $u(p)\in L_1$. Since $p$ is a maximum point, we get
    \begin{equation*}
        u_s(p)\in TL_1\cap\ker(dF).
    \end{equation*}
    Thus $ju_s=u_t$ is also in $\ker(dF)$ by equation \eqref{the Hopf Lemma Ess Assumption}, but this is a contradiction because we must have
    \begin{equation*}
        \frac{\partial(F\circ u)}{\partial t}(p) = dF(u_t)\neq 0
    \end{equation*}
    by Hopf's lemma, so we have
    \begin{equation*}
        \sup_{\Omega} F\circ u \leq \max (A, \sup_K F) = D(L_1, L_2, K)
    \end{equation*}
    Therefore, for every holomorphic strip $u$ with boundary on $(L_1, L_2)$, we get
    \begin{equation*}
        \sup_{\mathbb{R}\times [0,1]} r\circ u\leq D+M
    \end{equation*}
    where $r$ is the proper function defined in equation \eqref{The proper exhaustion}. Now we can compactify every moduli space $\mathcal{M}(x,y)$ by Gromov's compactness theorem because every element in $\mathcal{M}(x,y)$ has finite energy and its image is inside a compact set. Finally, we know $\pi_2(X), \pi_2(X,L_1)$ and $\pi_2(X,L_2)$ are trivial, so there are no sphere or disk bubbles. We conclude that $\partial^2=0$; see \cite{AudinDomanBookd} for more details.  
\end{proof}
Let $HF(L_1, L_2, H, J_t)$ denote the homology of $(CF(L_1, L_2, H, J_t), \partial)$.
\begin{prop}
    The homology $HF(L_1, L_2, H, J_t)$ does not depend on the choice of $(H, J_t)$.
\end{prop}
\begin{proof}
    Take two admissible pairs $(H_0, J_t^0)$ and $(H_1, J_t^1)$. We interpolate between these pairs by a path $(H_s, J_t^s)$ such that
    \begin{align*}
        &(H_s, J_t^s) = (H_0, J_t^0) \hspace{2mm} \forall s<<0, \\
        &(H_s, J_t^s) = (H_1, J_t^1) \hspace{2mm} \forall s>>0, 
    \end{align*}
    and $J_t^s = j$ outside of a compact set for every $(s,t)$.
    We consider maps $u\colon\mathbb{R}\times[0,1]\to X$ with 
    \begin{equation}
    \label{Interpolation equation}
        u_s+J_t^s(u)(u_t-X_{H_s}(u))=0,
    \end{equation}
    $u(s,0)\in L_1$, and $u(s,1)\in L_2$ for every $s$. If $u$ has finite energy, then
    \begin{align*}
        &\lim_{s\to-\infty}u(s,t) = \phi_{H_0}^t(x),\\
        &\lim_{s\to+\infty}u(s,t) = \phi_{H_1}^t(y)
    \end{align*}
    with $x\in L_1\cap (\phi_{H_0}^1)^{-1}(L_2)$ and $y\in L_1\cap (\phi_{H_1}^1)^{-1}(L_2)$.
    We denote the moduli space of solutions to \eqref{Interpolation equation} with two ends on $x(t) = \phi_{H_0}^t(x)$ and $y(t) = \phi_{H_1}^t(y)$ by $\mathcal{N}(x,y)$. For a generic choice of the interpolation $(H_s, J_t^s)$, these spaces are smooth manifolds, and we define a map $c\colon CF(L_1, L_2, H_0, J_t^0)\to CF(L_1, L_2, H_1, J_t^1)$ by
    \begin{equation*}
        C(x) = \sum_y \#\mathcal{N}_0(x,y)\cdot y
    \end{equation*}
    where $\mathcal{N}_0$ is the zero dimensional part of these moduli spaces. We can compactify these moduli spaces by an argument similar to Proposition \ref{boundary squared is zero}, and this shows $C$ is a chain map. Hence, we get a map
    \begin{equation*}
        C_{*}\colon HF(L_1, L_2, H_0, J_t^0)\to HF(L_1, L_2, H_1, J_t^1).
    \end{equation*}
    The next step is to prove the induced map $C_*$ is independent of the choice of interpolation $(H_s, J_t^s)$. Suppose $(H_{s,0}, J_{t,0}^s)$ and $(H_{s,1}, J_{t,1}^s)$ are two interpolations between $(H_0, J_t^0)$ and $(H_1, J_t^1)$. We take a path of interpolations $(H_{s,r}, J_{t,r}^s)$ parametrized by $r\in [0,1]$ between these two interpolations such that
    \begin{align*}
        &(H_{s,r}, J_{t,r}^s) = (H_0, J_t^0) \hspace{2mm} \forall s<<0, \\
        &(H_{s,r}, J_{t,r}^s) = (H_1, J_t^1) \hspace{2mm} \forall s>>0
    \end{align*}
    for every $r\in [0,1]$ and $J_{t,r}^s$ being equal to $j$ outside of a compact subset of $X$. Consider the solutions $u\colon\mathbb{R}\times [0,1]\to X$ of 
    \begin{equation}
    \label{Chain homotopy eq}
        u_s + J_{t,r}^s(u)(u_t-X_{H_{s,r}}(u)) = 0
    \end{equation}
    with finite energy. For $x\in L_1\cap (\phi_{H_0}^1)^{-1}(L_2)$ and $y\in L_1\cap (\phi_{H_1}^1)^{-1}(L_2)$, let
    \begin{equation*}
        \mathcal{K}(x,y)\coloneqq  \{(r,u)| \exists r\in [0,1] \hspace{1mm} \text{s.t. $u$ solves \eqref{Chain homotopy eq} with $\lim_{s\to-\infty}u =x(t)$, $\lim_{s\to\infty}u = y(t)$}\}.
    \end{equation*}
    We can assume $\mathcal{K}(x,y)$ is a smooth manifold for a generic choice of $(H_{s,r}, J_{t,r}^s)$, and define a map \break$\Phi\colon CF(L_1, L_2, H_0, J_t^0)\to CF(L_1, L_2, H_1, J_t^1)$ by
    \begin{equation*}
        \Phi(x) = \sum_y\#\mathcal{K}_0(x,y)\cdot y
    \end{equation*}
    where $\mathcal{K}_0$ denotes the zero dimensional parts of these moduli spaces. By an analogous argument, we have
    \begin{equation*}
        C^1-C^0 = \partial\circ \Phi+\Phi\circ\partial
    \end{equation*}
    where $C^i$ is the chain map associated to $(H_{s,i}, J_{t,i}^s)$ for $i=0,1$. Thus $C^1$ and $C^0$ are chain homotopic and their induced map on homology are equal. If we consider the constant interpolation from $(H_0, J_t^0)$ to itself, then the corresponding map on $HF(L_0, L_1, H_0, J_t^0)$ is the identity. Hence, if $C$ is the chain map corresponding to an interpolation from $(H_0, J_t^0)$ to $(H_1, J_t^1)$, and $D$ is the chain map corresponding to an interpolation from $(H_1, J_t^1)$ to $(H_0, J_t^0)$, then both $D\circ C$ and $C\circ D$ are chain homotopic to the identity maps of $CF(L_1, L_2, H_0, J_t^0)$ and $CF(L_1, L_2, H_1, J_t^1)$.   
\end{proof}
We denote $HF(L_1, L_2, H, J_t)$ by $HF(L_1, L_2)$ from now on.
\begin{cor}
\label{Invariance under compactly supported perturbations}
    Suppose $(L_1, L_2)$ is a suitable pair and $\phi\colon [0,1]\times X\to X$ is a compactly supported Hamiltonian isotopy. Then $(L_1, \phi_1(L_2))$ is also a suitable pair, and we have
    \begin{equation*}
        HF(L_1, L_2)\cong HF(L_1, \phi_1(L_2)).
    \end{equation*}
\end{cor}
\begin{proof}
    The symplectic action one form is exact on $\mathcal{P}(L_1, \phi_1(L_2))$ because it is exact on $\mathcal{P}(L_1, L_2)$, and $\phi$ is a Hamiltonian isotopy. The intersection between $L_1$ and $\phi_1(L_2)$ remains compact, and the last two assumptions in Definition \ref{Suitable pair Defintion} hold since $\phi$ has compact support.  
    Consider a smooth function $\beta\colon \mathbb{R}\to [0,1]$ so that $\beta(s) = 0$ for $s<<0$ and $\beta(s) = 1$ for $s>>0$. Choose an admissible pair $(H_0, J_t^0)$ for $(L_1, L_2)$ and an admissible pair $(H_1, J_t^1)$ for $(L_1, \phi(L_2))$. Interpolate between the two pairs as in the previous proposition, and consider finite energy solutions of 
  \begin{equation*}
      u_s + J_t^s(u)(u_t-X_{H_s}(u)) = 0,
  \end{equation*}
  $u\colon \mathbb{R}\times I\to X$,
  with boundary conditions
  \begin{equation*}
      u(s,0)\in L_1 \hspace{2mm} u(s,1)\in \phi_{\beta(s)}(L_2)\hspace{2mm} \forall s\in \mathbb{R}.
  \end{equation*}
  We can compactify the solutions of this equation using the same idea as in Proposition \ref{boundary squared is zero}. For a generic interpolation, the moduli spaces are smooth manifolds and we can count the zero dimensional parts of these manifolds as before. This gives us a quasi-isomorphism from $CF(L_1, L_2, H_0, J_t^0)$ to $CF(L_1, \phi^1(L_2), H_1, J_t^1)$. 
\end{proof}
\section{Vertical Lines and Clean Intersection}
\label{sec: vert lines and clean}

In this section, we consider four vertical lines 
\begin{equation*}
    \delta_i = \{x = \alpha_i\}\subset \mathbb{C}
\end{equation*}
for $1\leq i\leq 4$ as our periodic curves. First, we find when the Lagrangian pair $(L_{24}, L_{13}^\theta)$ corresponding to the quadruple $(\delta_1, \delta_2, \delta_3, \delta_4)$ is suitable. Then we compute the Floer homology $HF(L_{24}, L_{13}^\theta)$.
\begin{lem}
\label{Vertical lines are suitable}
    If $(\delta_1, \delta_2, \delta_3, \delta_4)$ are vertical lines with $\delta_i = \{x = \alpha_i\}$, then $(L_{24}, L_{13}^\theta)$ is a suitable pair iff
    \begin{equation}
    \label{line being suitable equation}
        t(\alpha_1-\alpha_2) + (1-t)(\alpha_3-\alpha_4)=0
    \end{equation}
\end{lem}
\begin{proof}
    The first two assumptions of Definition \ref{Suitable pair Defintion} follow from Lemma \ref{compactness of intersections} and Lemma \ref{Inclusion map induces iso on fundamental group}. The relation in equation \eqref{The area Lagrangian equation} is equivalent to equation \eqref{line being suitable equation} for the vertical lines $\delta_i$. Hence, the symplectic action one form on $\mathcal{P}(L_{24}, L_{13}^\theta)$ is exact iff equation \eqref{line being suitable equation} holds by Proposition \ref{The area Lagrangian proposition}. Take a positive number $N$ such that $\alpha_i\in [-N, N]$ for $1\leq i\leq 4$, then $g$ is bounded by $N$ on $L_{24}$ and $L_{13}$. It follows from equation \eqref{g does not change under phi} that $g$ is also bounded by $N$ on $L_{13}^\theta$. It remains to show the last assumption in Definition \ref{Suitable pair Defintion} holds. First, we prove this holds for $L_{24}$. Let $d = \vert\alpha_2-\alpha_4\vert$, and suppose $x\in L_{24}$ satisfies 
    \begin{equation}
    \label{distance greater than sth}
        F(x)>\frac{t(1-t)d^2}{2}.
    \end{equation}
    Take a point $(x,v)\in TL_{24}\cap \ker(dF)$, choose a point
    \begin{equation*}
        (z,w)\in\pi^{-1}(x)\subset\delta_2\times\delta_4,        
    \end{equation*}
    and $\tilde{v}=(v_1, v_2)\in T_{(z,w)}(\delta_2\times\delta_4)$ with $d\pi(\tilde{v})=v$. We get
    \begin{equation*}
        0 = dF(v) = v_1\cdot(z-w) -  v_2\cdot(z-w).
    \end{equation*}
    We can write $\tilde{v}=(r\sqrt{-1}, s\sqrt{-1})$ for two real numbers $r,s$ because $\delta_i$'s are vertical lines. Hence, we have
    \begin{equation*}
        0 = (r-s)\operatorname{Im}(z-w).
    \end{equation*}
    Observe that $\operatorname{Im}(z-w)\neq 0$ by equation \eqref{distance greater than sth}, so we must have $r=s$. Thus
    \begin{equation*}
        dF(jv) = v_1\cdot(z-w)-v_2\cdot(z-w) = (s-r)\operatorname{Re}(z-w) = 0.
    \end{equation*}
    A similar argument shows the last assumption of Definition \ref{Suitable pair Defintion} holds for $L_{13}$ with $d=\vert\alpha_1-\alpha_3\vert$. Since $L_{13}^\theta=\varphi_\theta(L_{13})$, $dF$ is invariant under $\varphi_\theta$, and $\varphi_\theta$ is holomorphic, we conclude that $L_{13}^\theta$ has the last condition as well, so the pair $(L_{24}, L_{13}^\theta)$ satisfies equation \eqref{the Hopf Lemma Ess Assumption} with 
    \begin{equation*}
        A = \frac{t(1-t)\max((\alpha_2-\alpha_4)^2, (\alpha_1-\alpha_3)^2)}{2}.
    \end{equation*}
\end{proof}
Now that we know $HF(L_{24}, L_{13}^\theta)$ is suitable if equation \eqref{line being suitable equation} holds, we are going to compute this Floer homology using the idea of clean intersection.

\begin{defi}
    We say two Lagrangians $L_1$ and $L_2$ have clean intersection if their intersection is a compact manifold $C$ in $X$, and we have
    \begin{equation*}
        T_xC = T_xL_1\cap T_xL_2
    \end{equation*}
    for every $x\in C$. This is a generalization of transverse intersection.
\end{defi} 

\begin{lem}
\label{Clean-Int-HF}
    Suppose $(L_1, L_2)$ is a suitable pair, and they intersect cleanly with $C=L_1\cap L_2$ a closed submanifold of $X$. Then
    \begin{equation*}
        HF(L_1, L_2)\cong H^*(C;\mathbb{F}_2).
    \end{equation*}
\end{lem}
\begin{proof}
    It follows from \cite[Theorem 3.4.11]{6} that there is a spectral sequence from $H^*(C)$ to $HF(L_1, L_2)$. The higher differentials in this spectral sequence involve counting holomorphic strips with two ends on $C$, but there are no holomorphic strips with two ends on $C$ because the symplectic action one form is exact in our case; in particular, suppose $u$ is a non-constant holomorphic strip with two ends $x,y\in L_1\cap L_2$. Then we must have $\mathcal{A}(x)=\mathcal{A}(y)$ because $C$ is a closed critical submanifold of the symplectic action, and we get
    \begin{equation*}
        E(u)=\mathcal{A}(y)-\mathcal{A}(x) = 0
    \end{equation*}
    which is a contradiction;
    we refer the unfamiliar reader to \cite[Chapter 10]{ CleanFloerHomology} for more details. 
\end{proof}
\begin{lem}
\label{Vertical lines meet cleanly}
    Suppose $\delta_i = \{x=\alpha_i\}$ is a vertical line for each $1\leq i\leq 4$, and we have the relation in equation \eqref{line being suitable equation} among $\alpha_i's$. Then $L_{24}$ and $L_{13}^\theta$ intersect cleanly along a circle.
\end{lem}
\begin{proof}
    Since $\pi\colon\mathbb{C}^2\to X$ is a local diffeomorphism, it suffices to prove $\delta_2\times\delta_4$ meets $\varphi_\theta(\delta_1\times\delta_3)$ cleanly in a line. Then the quotient of this line by translation is the desired circle of intersection between $L_{24}$ and $L_{13}^\theta$ in $X$. 
    Take a pair $(z,w) = (\alpha_1, r_1,\alpha_3, r_3)\in \delta_1\times\delta_3$, and assume $\varphi_\theta(z,w)\in \delta_2\times\delta_4$. Then we get
    \begin{equation*}
        \operatorname{Re}(e^{i\theta}(z-w)) = \cos\theta(\alpha_1-\alpha_3)+\sin\theta(r_3-r_1)=\alpha_2-\alpha_4,
    \end{equation*}
    so we have
    \begin{equation*}
        r_3 = r_1 + \cot\theta(\alpha_3-\alpha_1)+\frac{(\alpha_2-\alpha_4)}{\sin\theta}.
    \end{equation*}
    Note that since we have the relation in equation \eqref{line being suitable equation}, we get
    \begin{equation*}
        \varphi_\theta\big((\alpha_1,r,\alpha_3, r+\cot\theta(\alpha_3-\alpha_1)+\frac{(\alpha_2-\alpha_4)}{\sin\theta})\big)\in \delta_2\times\delta_4.
    \end{equation*}
    for every $r\in\mathbb{R}$; this is the line of intersection between $\varphi_\theta(\delta_1\times\delta_3)$ and $\delta_2\times\delta_4$. See the following figure for the case of a rectangle inscribed in four vertical lines.
    Denote this line by $l$; it remains to show
    \begin{equation}
    \label{equality of tangent bundle intersections}
        T_{(z,w)}l = T_{(z,w)}(\delta_2\times\delta_4)\cap T_{(z,w)}(\varphi_\theta(\delta_1\times\delta_3))
    \end{equation}
    for every $(z,w)\in l$.
\begin{figure}[h]
\begin{center}
\begin{tikzpicture}[scale=1]

\foreach \x in {0,1,2,3}{
    \draw[blue, thick] (\x,-1.2) -- (\x,3.2);
}

\node[below] at (0,-1.2) {$\delta_1$};
\node[below] at (1,-1.2) {$\delta_2$};
\node[below] at (2,-1.2) {$\delta_4$};
\node[below] at (3,-1.2) {$\delta_3$};

\coordinate (A) at (0,0);
\coordinate (B) at (1,{2-2*sqrt(2)});
\coordinate (C) at (3,{3-sqrt(2)});
\coordinate (D) at (2,{1+sqrt(2)});

\draw[thick] (A) -- (B) -- (C) -- (D) -- cycle;

\draw[red, thick] (A) -- (C);
\draw[red, thick] (B) -- (D);

\fill (A) circle (1pt);
\fill (B) circle (1pt);
\fill (C) circle (1pt);
\fill (D) circle (1pt);

\end{tikzpicture}
        \end{center}
        \caption{$t=\frac{1}{2},\theta=\frac{\pi}{4}$}
        \label{fig:placeholder}
\end{figure}
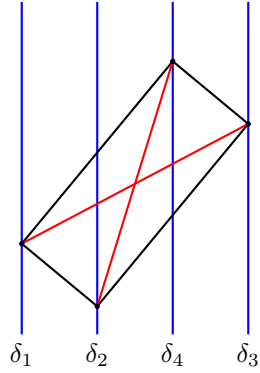
    
    First, we do a change of coordinates $\psi\colon \mathbb{C}^2\to\mathbb{C}^2$ by
    \begin{equation*}
        \psi(z,w) = (tz+(1-t)w, z-w),
    \end{equation*}
    which makes the computations more tractable. Then we get
    \begin{equation*}
        R_\theta\coloneqq\psi\circ\varphi_\theta\circ\psi^{-1}(z,w) = (z, e^{i\theta}w).
    \end{equation*}
    Consider the coordinates $(x_1, y_1, x_2, y_2)$ on $\mathbb{C}^2$. We have
    \begin{align*}
        d\psi(T_{(z,w)}(\delta_2\times\delta_4))=\langle t\partial y_1+\partial y_2, (1-t) \partial y_1-\partial y_2\rangle = \langle\partial y_1, \partial y_2\rangle 
    \end{align*}
    for every $(z,w)\in \delta_2\times\delta_4$. Similarly, we compute
    \begin{equation*}
        dR_\theta\circ d\psi(T_{(z,w)}(\delta_1\times\delta_3))=\langle\partial y_1, \cos\theta\partial y_2-\sin\theta\partial x_2\rangle
    \end{equation*}
    for every $(z,w)\in \delta_1\times\delta_3$. Therefore,
    \begin{equation*}
        d\psi\big(T_{(z,w)}(\delta_2\times\delta_4)\cap T_{(z,w)}(\varphi_\theta(\delta_1\times\delta_3))\big) = \langle\partial y_1\rangle
    \end{equation*}
    because we assumed $\theta\in(0,\pi)$. We conclude that equation \eqref{equality of tangent bundle intersections} holds since $d\psi(T_{(z,w)}l)=\langle\partial y_1\rangle$.
\end{proof}
\begin{prop}
\label{HF is the circle homology}
    Suppose $(\delta_1,\delta_2,\delta_3,\delta_4)$ are a quadruple of vertical lines satisfying equation \eqref{line being suitable equation}. Then
    \begin{equation*}
        HF(L_{24}, L_{13}^\theta)\cong \mathbb{F}_2^2.
    \end{equation*}
\end{prop}
\begin{proof}
    By Lemma \ref{Vertical lines meet cleanly}, we know $L_{24}\cap L_{13}^\theta$ is a circle and this intersection is clean. Thus we get
    \begin{equation*}
        HF(L_{24}, L_{13}^\theta)\cong H^*(S^1)
    \end{equation*}
    by Lemma \ref{Clean-Int-HF}.
\end{proof}
\section{Periodic Curves and Suitable Lagrangians}
\label{sec: periodic curves and suitable lags}

Let $(\gamma_1,\gamma_2,\gamma_3,\gamma_4)$ be a quadruple of smooth periodic curves, and construct the Lagrangian pair $(L_{24}, L_{13}^\theta)$ in $(X,\omega)$. We know the inscriptions of type $(t,\theta)$ trapezoids are in one-to-one correspondence with the intersection $L_{24}\cap L_{13}^\theta$. We want to show this set is non-empty if \eqref{The area Lagrangian equation} holds. The idea is to use Floer homology; however, the pair $(L_{24}, L_{13}^\theta)$ is not suitable for every choice of $\gamma_i$'s, and we cannot define Floer homology in this case. We will fix this issue by perturbing this pair to another pair $(\Bar{L}_{24}, \Bar{L}_{13}^\theta)$ such that
\begin{equation*}
    L_{24}\cap L_{13}^\theta = \Bar{L}_{24}\cap\Bar{L}_{13}^\theta,
\end{equation*}
and $(\Bar{L}_{24}, \Bar{L}_{13}^\theta)$ is suitable.
\begin{lem}
\label{Bounded Hamiltonian perturbations}
    Suppose $(\gamma_1,\gamma_2,\gamma_3,\gamma_4)$ is a quadruple of smooth periodic curves inside the strip $[-N,N]\times\mathbb{R}$ in the complex plane.
    Then there exist two time-dependent Hamiltonian $G_1, G_2\colon X\times I\to \mathbb{R}$ and a quadruple of vertical lines $(\delta_1,\delta_2, \delta_3, \delta_4)$ in $[-N,N]\times\mathbb{R}$ such that
    \begin{align*}
        &\phi_{G_1}^1(K_{24}) = L_{24},\\
        & \phi_{G_2}^1(K_{13}^\theta) = L_{13}^\theta,
    \end{align*}
    where $(K_{24}, K_{13}^\theta)$ and $(L_{24}, L_{13}^\theta)$ are the Lagrangian pairs associated to $\delta_i$'s and $\gamma_i$'s respectively as in Section \ref{sec: symp setting}. Moreover, we can choose $G_i$'s so that they are supported in the set $g^{-1}([-2N, 2N])$,\break and $\Vert G_i\Vert_{C^1(X\times [0,1])}<\infty$ for each $i$.
\end{lem}
\begin{proof}
    Consider a smooth periodic curve $\gamma$ in $\mathbb{C}$; let $\tilde{\gamma}$ be its image in the cylinder $S\cong \mathbb{R}\times S^1$. There is an $\alpha$ such that the signed area between $\tilde{\gamma}$ and $\{\alpha\}\times S^1$ is zero. Hence, we can find a time-dependent Hamiltonian $\tilde{H}\colon S\times [0,1]\to \mathbb{R}$ supported in $[-N,N]\times S^1\times[0,1]$ with $\phi_{\tilde{H}}^1(\{\alpha\}\times S^1)=\tilde{\gamma}$. This Hamiltonian lifts to a Hamiltonian $H$ over $\mathbb{C}$ that is invariant under translation by $\sqrt{-1}$ with $\phi_H^1(\{\alpha\}\times\mathbb{R})=\gamma$, and is supported in $[-N,N]\times\mathbb{R}\times [0,1]$. Note that $H$ has bounded $C^1$-norm because it is the periodic lift of a compactly supported Hamiltonian in $S\times [0,1]$. Apply this procedure to $\gamma_2$ and $\gamma_4$ to find periodic Hamiltonian functions $H_2, H_4$ and vertical lines $\delta_2 = \{\alpha_2\}\times\mathbb{R}, \delta_4 = \{\alpha_4\}\times\mathbb{R}$ with
    \begin{equation*}
        \phi_{H_2}(\delta_2)=\gamma_2, \hspace{1mm} \phi_{H_4}(\delta_4)=\gamma_4.
    \end{equation*}
    Define a Hamiltonian
    \begin{equation*}
        G = tH_2+(1-t)H_4
    \end{equation*}
    on $\mathbb{C}^2$. This Hamiltonian is invariant under $T$, so it reduces to a Hamiltonian on $X$ that we denote by $G$ as well. Now we get $\phi_{G}^1(K_{24}) = L_{24}$, but this is not supported in $g^{-1}([-2N,2N])$. However, observe that $\phi_G^t(K_{24})$ remains inside $g^{-1}([-N,N])$ for $t\in[0,1]$ by construction; thus we can just put $G_1 = \chi\circ g\cdot G$ where $\chi$ is a smooth function supported in $[-2N, 2N]$ with $\chi|_{[-N,N]} = 1$. Similarly, one can find a Hamiltonian $\hat{G}$ supported in $g^{-1}([-2N, 2N])$ with bounded $C^1$-norm such that $\phi^1_{\hat{G}}(K_{13}) = L_{13}$. We write
    \begin{equation*}
        L_{13}^\theta = \varphi_\theta(L_{13}) = \varphi_\theta(\phi^1_{\hat{G}}(K_{13})) = \varphi_\theta\circ\phi_{\hat{G}}^1\circ\varphi_{-\theta}(K_{13}^\theta).
    \end{equation*}
    Hence, we can put $G_2 = \hat{G}\circ\varphi_{-\theta}$. This Hamiltonian remains supported in $g^{-1}([-2N, 2N])$ because $g$ is invariant under $\varphi_{-\theta}$. It also remains $C^1$-bounded since $d\varphi_{-\theta}$ has finite norm.
\end{proof}
Consider the function $p\colon\mathbb{C}^2\to \mathbb{C}$ defined by $p(z,w)=z-w$. We have $p\circ T=p$, and this descends to a function on $X$ that we denote by $p$ as well. Note that $\Vert p(x)\Vert + g^2(x)$ is another proper function on $X$. 
\begin{prop}
\label{Change to Suitable Pairs}
    Assume $(\gamma_1,\gamma_2,\gamma_3,\gamma_4)$ is a quadruple of smooth periodic curves in $\mathbb{C}$, and they satisfy equation \eqref{The area Lagrangian equation}. Then there exists a suitable pair $(\Bar{L}_{24}, \Bar{L}_{13}^\theta)$ in $X$ such that $L_{24}\cap L_{13}^\theta = \Bar{L}_{24}\cap\Bar{L}_{13}^\theta$.
\end{prop}
\begin{proof}
    Find a large $N$ so that each $\gamma_i$ is inside $[-N,N]\times\mathbb{R}$.
    By Lemma \ref{Bounded Hamiltonian perturbations}, there are a quadruple of vertical lines $(\delta_1, \delta_2,\delta_3,\delta_4)$ in $[-N,N]\times\mathbb{R}$ and two time-dependent Hamiltonian $G_1, G_2$ with $\Vert G_i\Vert_{C^1(X\times[0,1])}<\infty$ for each $i$ such that
    \begin{equation*}
        \phi_{G_1}^1(K_{24}) = L_{24},\hspace{1mm} \phi_{G_2}^1(K_{13}^\theta) = L_{13}^\theta.
    \end{equation*}
    Notice that the symplectic action one form on $\mathcal{P}(K_{24}, K_{13}^\theta)$ is exact because it is exact on $\mathcal{P}(L_{24}, L
    _{13}^\theta)$ by equation \eqref{The area Lagrangian equation}, and this pair is a Hamiltonian perturbation of $(K_{24}, K_{13}^\theta)$; we refer the unfamiliar reader to \cite[Chapter 3]{14}. 
    Let $A$ be a sufficiently large number, and consider a non-negative smooth function $h$ on $\mathbb{R}$ with
    \begin{equation*}
        h(r)=
        \begin{cases}
             1 &r\leq A,\\
             0 &r>>0,
        \end{cases}
    \end{equation*}
    and $\vert h'(r)\vert\leq t(1-t)$.
    We define a Hamiltonian function $H_i$ by
    \begin{equation}
    \label{Definition of the compactly supported perturbations}
        H_i = h(\Vert p(x)\Vert)\cdot G_i
    \end{equation}
    for each $i=1,2$. Both of these Hamiltonian functions are compactly supported, since each $G_i$ is supported in $g^{-1}([-2N, 2N])$, and $h(\Vert p(x)\Vert)$ vanishes if $\Vert p(x)\Vert>>0$. We let
    \begin{equation*}
        \Bar{L}_{24} \coloneqq \phi_{H_1}(K_{24}),\hspace{2mm} \Bar{L}_{13}^\theta \coloneqq \phi_{H_2}(K_{13}^\theta).
    \end{equation*}
    The pair $(\Bar{L}_{24}, \Bar{L}_{13}^\theta)$ is suitable because it is a compactly supported perturbation of the suitable pair $(K_{24}, K_{13}^\theta)$. It remains to show $L_{24}\cap L_{13}^\theta = \Bar{L}_{24}\cap\Bar{L}_{13}^\theta$. Denote the function $h(\Vert p(x)\Vert)$ by $V$; we have
    \begin{equation*}
        X_{H_i} = X_V\cdot G_i+V\cdot X_{G_i}
    \end{equation*}
    for each $i=1,2$. We can write
    \begin{equation*}
        X_V(x) = \frac{h'(\Vert p(x)\Vert)}{\Vert p(x)\Vert}\cdot\Big(\frac{\sqrt{-1}p(x)}{t}, \frac{\sqrt{-1}p(x)}{t-1}\Big),
    \end{equation*}
    where we identify $T_xX\cong\mathbb{C}^2$. Note
    \begin{align*}
        & dp(\sqrt{-1}(p(x),0)) = \sqrt{-1}p(x), \\
        &dp((0, \sqrt{-1}p(x))) = -\sqrt{-1}p(x).
    \end{align*}
    Therefore,
    \begin{equation*}
        \Vert dp(X_V)\Vert = \frac{h'(\Vert p(x)\Vert)}{\Vert p(x)\Vert}\cdot \frac{\Vert p(x)\Vert}{t(1-t)}\leq 1.
    \end{equation*}
    This shows
    \begin{equation*}
        \Vert dp(X_{H_i})\Vert \leq \Vert dp(X_V)\Vert\cdot\vert G_i\vert + V\cdot \Vert dp(X_{G_i})\Vert\leq 2\Vert G_i\Vert_{C^1(X\times [0,1])}
    \end{equation*}
    for each $i$. Let $M = 2\max(\Vert G_1\Vert_{C^1(X\times [0,1])}, \Vert G_2\Vert_{C^1(X\times [0,1])})$.
    Now take a point $y = \phi_{H_i}^s(x)$ for time $s\in[-1,1]$; we have
    \begin{equation}
    \label{Bounding the p distance}
        \Vert p(y) - p(x)\Vert\leq \int_0^s \Vert dp(X_{H_i})\Vert dt\leq M\cdot\vert s\vert\leq M. 
    \end{equation}
    Thus
    \begin{equation*}
        \Vert p(\phi^s_{H_i}(x))\Vert\leq \Vert p(x)\Vert + M
    \end{equation*}
    for each $i$ and $s\in [-1,1]$. Hence, if $\Vert p(x)\Vert\leq A-M$, then $\Vert p(\phi^s_{H_i}(x))\Vert\leq A$. By equation \eqref{Definition of the compactly supported perturbations}, $H_i$ and $G_i$ agree on the points with $\Vert p(x)\Vert\leq A$, so
    \begin{equation*}
        \phi_{H_i}^s(x) = \phi_{G_i}^s(x)
    \end{equation*}
    for every $s\in[-1,1]$ and each $i$ if $\Vert p(x)\Vert\leq A-M$.
    We conclude that
    \begin{equation}
    \label{p-bounded equality of Lagrangians}
        L_{24}\cap\{x:\Vert p(x)\Vert\leq A-M\} = \Bar{L}_{24}\cap\{x:\Vert p(x)\Vert\leq A-M\}. 
    \end{equation}
    The same claim holds for $L_{13}^\theta$ and $\Bar{L}_{13}^\theta$. The next step is to bound $\Vert p(x)\Vert$ over $\Bar{L}_{24}\cap \Bar{L}_{13}^\theta$. Take a point $x$ in this intersection; there are $y_1\in K_{24}$ and $y_2\in K_{13}^\theta$ such that
    \begin{equation*}
        x = \phi^1_{H_1}(y_1)=\phi^1_{H_2}(y_2).
    \end{equation*}
    We have
    \begin{align*}
        & \vert\operatorname{Re}(p(x)-p(y_1))\vert\leq \Vert p(x)-p(y_1)\Vert\leq M,\\
         \vert\operatorname{Re}&(e^{-i\theta}\cdot(p(x)-p(y_2)))\vert\leq \Vert p(x)-p(y_2)\Vert\leq M
    \end{align*}
    by equation \eqref{Bounding the p distance}. Observe that both $\vert\operatorname{Re}(p(y_1))\vert$ and $\vert\operatorname{Re}(e^{-i\theta}\cdot p(y_2))\vert$ are less than $2N$ because we assumed the quadruple $(\delta_1, \delta_2,\delta_3,\delta_4)$ is inside $[-N, N]\times\mathbb{R}$. Therefore, we get
    \begin{align*}
        & \vert \operatorname{Re}(p(x))\vert\leq M+2N, \\
        \vert\operatorname{Re}&(e^{-i\theta}\cdot p(x))\vert\leq M+2N.
    \end{align*}
    Combining these two equations, we have
    \begin{equation*}
        \Vert p(x)\Vert\leq (M+2N)\cdot\sqrt{1+\big(\frac{1+\vert\cos\theta\vert}{\sin\theta}\big)^2} \coloneqq B.
    \end{equation*}
    We conclude that
    \begin{equation}
    \label{int one being bounded}
        \Bar{L}_{24}\cap\Bar{L}_{13}^\theta\subset \{x: \Vert p(x)\Vert\leq B\}.
    \end{equation}
    We also know
    \begin{equation}
    \label{int two being bounded}
        L_{24}\cap L_{13}^\theta\subset \{x: \Vert p(x)\Vert\leq b(N,t,\theta)\},
    \end{equation}
    where $b$ is the constant determined in Lemma \ref{compactness of intersections}. Now choose $A$ large enough such that
    \begin{equation*}
        \max(b,B)< A-M.
    \end{equation*}
    With this choice, we conclude 
    \begin{equation*}
        L_{24}\cap L_{13}^\theta = \Bar{L}_{24}\cap\Bar{L}_{13}^\theta
    \end{equation*}
    by equations \eqref{p-bounded equality of Lagrangians}, \eqref{int one being bounded} and \eqref{int two being bounded}.
\end{proof}
\begin{cor}
\label{non-vanishing HF}
    For the pair $(\Bar{L}_{24}, \Bar{L}_{13}^\theta)$, we have
    \begin{equation*}
        HF(\Bar{L}_{24}, \Bar{L}_{13}^\theta)\cong\mathbb{F}_2^2.
    \end{equation*}
\end{cor}
\begin{proof}
    The pair $(\Bar{L}_{24}, \Bar{L}_{13}^\theta)$ is a compactly supported perturbation of the pair $(K_{24}, K_{13}^\theta)$. Thus the claim follows from Corollary \ref{Invariance under compactly supported perturbations} and Proposition \ref{HF is the circle homology}.
\end{proof}
\section{The Four Periodic Inscription Theorem}
\label{sec: main section}

\begin{proof}[\textbf{Proof of Theorem \ref{Main Thm}}]
Consider the quadruple $(\gamma_1,\gamma_2,\gamma_3,\gamma_4)$, and assume they satisfy the area relation in \eqref{t-area equality relation} with 
\begin{equation*}
    \gamma_1(\mathbb{R}) \cap \gamma_2(\mathbb{R})\cap \gamma_3(\mathbb{R})\cap \gamma_4(\mathbb{R})=\emptyset.
\end{equation*}
We construct the two Lagrangians $L_{24}$ and $L_{13}^\theta$ corresponding to this quadruple in $X$. Since we know the intersection of these four curves is empty, we have
\begin{equation*}
    \frac{\mathcal{I}(\gamma_1,\gamma_2,\gamma_3,\gamma_4)}{\text{Translation by $\sqrt{-1}$}} = L_{24}\cap L_{13}^\theta,
\end{equation*}
and it suffices to show $L_{24}\cap L_{13}^\theta$ is non-empty. By Proposition \ref{Change to Suitable Pairs}, we know there is a suitable pair of Lagrangians $(\Bar{L}_{24}, \Bar{L}_{13}^\theta)$ in $X$ such that $L_{24}\cap L_{13}^\theta = \Bar{L}_{24}\cap\Bar{L}_{13}^\theta$, and we have $ HF(\Bar{L}_{24}, \Bar{L}_{13}^\theta)\cong\mathbb{F}_2^2$ by Corollary \ref{non-vanishing HF}. Therefore, the set $\mathcal{I}(\gamma_1,\gamma_2,\gamma_3,\gamma_4)$ must be non-empty.
\end{proof}
\begin{rem}
    If the intersection between $L_{24}$ and $L_{13}^\theta$ is transverse, then we have at least two inscriptions that are not related under translation because
    \begin{equation*}
        \vert L_{24}\cap L_{13}^\theta\vert=\vert \Bar{L}_{24}\cap \Bar{L}_{13}^\theta\vert\geq \operatorname{rk}(HF(\Bar{L}_{24}, \Bar{L}_{13}^\theta))\geq 2.
    \end{equation*}
\end{rem}
\begin{proof}[\textbf{Proof of Theorem \ref{main thm for continuous curves}}]
    Take a quadruple of continuous curves $(\gamma_1,\gamma_2,\gamma_3,\gamma_4)$ with zero Lebesgue measure, and suppose they satisfy the area relation in \eqref{t-area equality relation}. Consider a sequence of smooth quadruples $(\gamma_1^n,\gamma_2^n,\gamma_3^n,\gamma_4^n)$ that are converging uniformly to $(\gamma_1,\gamma_2,\gamma_3,\gamma_4)$. Translate each $\gamma_4^n$ by a horizontal vector $r_n$ such that each quadruple $(\gamma_1^n,\gamma_2^n,\gamma_3^n,\gamma_4^n+r_n)$  satisfies the area relation in \eqref{t-area equality relation}. Notice that $r_n\to 0$ because 
    \begin{equation*}
        tA(\gamma_1^n,\gamma_2^n)+(1-t)A(\gamma_3^n,\gamma_4^n)\to tA(\gamma_1,\gamma_2)+(1-t)A(\gamma_3,\gamma_4)=0.
    \end{equation*}
    Hence, $\gamma_4^n+r_n$ also converges to $\gamma_4$. We know there is a type $(t,\theta)$ trapezoid $Q_n$ inscribed in each quadruple $(\gamma_1^n,\gamma_2^n,\gamma_3^n,\gamma_4^n+r_n)$ by Theorem \ref{Main Thm}. Moreover, we can assume each $Q_n$ is inside a bounded box by an argument similar to Lemma \ref{compactness of intersections}. Therefore, we can assume the four vertices of $Q_n$ are converging to four points $p_1, p_2, p_3,$ and $p_4$ with $p_i\in \gamma_i(\mathbb{R})$ for each $i$. It is possible that all four points are the same, and this gives an intersection point among the four curves, but we assumed their intersection is empty. We conclude that the quadrilateral formed by $p_i$'s is a type $(t,\theta)$ trapezoid inscribed in $(\gamma_1,\gamma_2,\gamma_3,\gamma_4)$. 
\end{proof}
\begin{cor}
\label{main cor}
    Suppose $(\gamma_1,\gamma_2,\gamma_3,\gamma_4)$  is a quadruple of continuous periodic curves in $\mathbb{C}$ with zero Lebesgue measure, and they satisfy the area relation in \eqref{t-area equality relation}. Then at least one of the following is true.
    \begin{enumerate}
        \item $(\gamma_1,\gamma_2,\gamma_3,\gamma_4)$  inscribes type $(t,\theta)$ trapezoid for every $\theta\in (0,\pi)$.
        \item The set 
        \begin{equation*}
            \gamma_1(\mathbb{R})\cap \gamma_2(\mathbb{R})\cap \gamma_3(\mathbb{R})\cap \gamma_4(\mathbb{R})
        \end{equation*}
        is non-empty.
    \end{enumerate}
\end{cor}
\begin{proof}
    This follows from the argument in the proof of Theorem \ref{main thm for continuous curves}.
\end{proof}

If we have a triple of smooth periodic curves $(\gamma_1,\gamma_2,\gamma_3)$, and their images $(\tilde{\gamma}_1, \tilde{\gamma}_2, \tilde{\gamma}_3)$ in $\mathbb{R}\times S^1$ have
\begin{equation*}
    t(\int_{\tilde{\gamma}_1}xdy-\int_{\tilde{\gamma}_2}xdy)+ (1-t)\int_{\tilde{\gamma}_3}xdy=0,
\end{equation*}
then the quadruple $(\gamma_1,\gamma_2,\gamma_3, \{x=0\})$ satisfies the area relation in \eqref{t-area equality relation}, so we can apply Corollary \ref{main cor} to these four curves. This gives us the following triangular inscription for the case of three curves, which is a generalization of \cite[Conjecture 6.1]{1}. 

\begin{thm}
    Let $\Delta$ be a triangle with angles $\theta_1\leq \theta_2\leq \theta_3$. Suppose $\gamma_1,\gamma_2$ are continuous periodic curves with zero Lebesgue measure, and $\gamma_3$ is a rectifiable periodic curve. Moreover, assume we have
    \begin{equation}
        \label{triangle area relation}
        \sin(\theta_3-\theta_1)A(\gamma_1,\gamma_2)+\sin(\theta_2)\int_{\tilde{\gamma}_3}xdy=0,
    \end{equation}
    and the intersection of these three curves is empty. Then there exist points $(p_1, p_2, p_3)$ in $\mathbb{C}$ with $p_i\in \gamma_i(\mathbb{R})$ for every $i$, the triangle formed by these points is similar to $\Delta$, and we have
    \begin{equation}
    \label{triangle real equation}
        \sin(\theta_3-\theta_1)\cdot(\operatorname{Re}(p_1)-\operatorname{Re}(p_2))+\sin(\theta_2)\operatorname{Re}(p_3)=0.
    \end{equation}
\end{thm}
\begin{proof}
    First we prove the claim for when $\theta_3\neq \theta_1$; this is when $\Delta$ is not an equilateral triangle. Choose $t\in(0,1)$ such that
    \begin{equation*}
        \frac{t}{1-t} = \frac{\sin(\theta_3-\theta_1)}{\sin(\theta_2)}.
    \end{equation*}
    It suffices to prove the result for smooth curves, since the case of continuous curves follows by a convergence argument similar to the proof of Theorem \ref{main thm for continuous curves}. Suppose $\gamma_i$'s are smooth; then we know the quadruple $(\gamma_1,\gamma_2,\gamma_3,\{x=0\})$ satisfies the area relation in \eqref{t-area equality relation} by equation \eqref{triangle area relation}. There is a type $(t, \pi-2\theta_1)$ trapezoid inscribed in $(\gamma_1,\gamma_2,\gamma_3,\{x=0\})$ by Theorem \ref{Main Thm}. Notice that the three vertices of this trapezoid that lie on $(\gamma_1,\gamma_2,\gamma_3)$ form a triangle similar to $\Delta$, and they satisfy \eqref{triangle real equation} because the fourth vertex lies on $\{x=0\}$. Now assume $\Delta$ is an equilateral triangle. Approximate $\Delta$ with a sequence of triangles $\Delta_n$ with angles $(\theta_1^n, \theta_2^n,\theta_3^n)$ and $\theta_1^n\neq \theta_3^n$. We can find a sequence of real numbers $r_n\to 0$ such that $(\gamma_1,\gamma_2,\gamma_3+r_n)$ satisfies \eqref{triangle area relation} with angles $(\theta_1^n, \theta_2^n,\theta_3^n)$. Thus  
    there is an inscription of $\Delta_n$ in $(\gamma_1,\gamma_2,\gamma_3+r_n)$ satisfying \eqref{triangle real equation}. Furthermore, we can assume there is a vertex of $\Delta_n$ that lies in a bounded rectangle in $\mathbb{C}$ by translating each $\Delta_n$ with an integer multiple of $\sqrt{-1}$.
    We can show all of these triangles have bounded diameter by an argument similar to Lemma \ref{bounded band similarity type}. Hence, each $\Delta_n$ lies in a bounded box, and a convergent subsequence of $\Delta_n$'s gives an inscription of $\Delta$ in $(\gamma_1,\gamma_2,\gamma_3)$.
\end{proof}
\begin{cor}[Conjecture 6.2 in \cite{1}]
    Consider three rectifiable periodic curves $\gamma_1,\gamma_2,$ and $\gamma_3$ in $\mathbb{C}$ if
    \begin{equation}
    \label{triple area in the special case}
        \int_{\tilde{\gamma}_1}xdy + \int_{\tilde{\gamma}_2}xdy+ \int_{\tilde{\gamma}_3}xdy = 0,
    \end{equation}
    then there are three points $p_1,p_2,$ and $p_3$ in the complex plane with $p_i\in\gamma_i(\mathbb{R})$, and
    \begin{align*}
        \operatorname{Re}&(p_1)+\operatorname{Re}(p_2)+\operatorname{Re}(p_3) = 0, \\
        &\operatorname{Im}(p_1)=\operatorname{Im}(p_2)=\operatorname{Im}(p_3).
    \end{align*}
\end{cor}
\begin{proof}
    We define three linear maps $R_1,R_2,$ and $R_3$ on $\mathbb{C}$ by
    \begin{align*}
        & R_1(x,y) = (x, y-x),\\
        & R_2(x,y) = (-x,y),\\
        & R_3(x,y) = (x, x+y).
    \end{align*}
    For each $i$, $R_i(\gamma_i)$ is a rectifiable periodic curve, and the quadruple $(\{x=0\}, R_1(\gamma_1), R_2(\gamma_2), R_3(\gamma_3))$ satisfies the area criterion in \eqref{the one-half case of area inequality} by equation \eqref{triple area in the special case}. Therefore, we can apply Corollary \ref{main cor} to these four curves. If there is an intersection point $p$ among these curves, then $\operatorname{Re}(p)=0$, and $p_i = R_i^{-1}(p)$ are the desired points. If the intersection is empty, we know there is a square inscribed in $(\{x=0\}, R_1(\gamma_1), R_2(\gamma_2), R_3(\gamma_3))$. Let $q_1, q_2,$ and $q_3$ be the three vertices of this square that lie on $(R_1(\gamma_1),R_2(\gamma_2),R_3(\gamma_3))$. It is straightforward to check the points $p_i = R_i^{-1}(q_i)$ are the claimed points.
\end{proof}

\newpage

\bibliographystyle{amsplain}
\bibliography{ref}

\begin{flushleft}
    Boston College. Massachusetts, USA.\\
    naserisa@bc.edu 
\end{flushleft}
\end{document}